\documentclass[10pt]{article}
\usepackage{amsfonts}
\usepackage{amsthm}
\usepackage{amssymb}
\usepackage{graphicx, enumerate}
\usepackage{amsmath}
\usepackage{mathtools}
\usepackage{latexsym}
\usepackage{longtable}
\usepackage{tabularx}
\usepackage{amsmath}
\usepackage{amsfonts}
\usepackage{amsthm}
\usepackage{setspace}
\usepackage{graphicx}
\usepackage{float}
\usepackage{rotating}
\usepackage{tikz}
\usepackage{verbatim}
\usepackage{caption}
\usepackage{subcaption}

\usepackage{mathtools}
\usepackage{amsmath,amssymb,amsbsy}
\usepackage{graphicx}
\usepackage{multirow}
\usepackage{amsfonts}
\usepackage{amsthm}
\usepackage{amssymb}
\usepackage{graphicx, enumerate}
\usepackage{amsmath}
\usepackage{latexsym}
\usepackage{longtable}
\usepackage{tabularx}
\usepackage{amsmath}
\usepackage{amsfonts}
\usepackage{amsthm}
\usepackage{setspace}
\usepackage{graphicx}
\usepackage{float}
\usepackage{rotating}
\usepackage{tikz}
\usepackage{verbatim}
\usepackage{soul}
\usepackage{xcolor}

\usepackage[normalem]{ulem}

\usepackage[main=english,czech,slovak]{babel}

\definecolor{darkpastelgreen}{rgb}{0.01, 0.75, 0.24}
\definecolor{blue-violet}{rgb}{0.54, 0.17, 0.89}

\usepackage{soul}
\usepackage{ulem}
\newtheorem{theorem}{Theorem}[section]

\newtheorem{prop}[theorem]{Proposition}
\newtheorem{obs}[theorem]{Observation}

\theoremstyle{definition}
\newtheorem{definition}[theorem]{Definition}

\theoremstyle{definition}
\newtheorem{const}[theorem]{Construction}
\newtheorem*{prb*}{Open Problem}
\newtheorem{prb}{Open Problem}
\newtheorem{exm}[theorem]{Example}

\def\Spec{\mathrm{Spec}}

\def\tilmu{\tilde{\mu}}
\vspace{5cm}
\begin{document}

\textwidth4.5true in
\textheight7.2true in

\title
{\Large \sc \bf {Semi-magic dihedral squares}}
\date{}
\author{{{Sylwia Cichacz$^{1}$, Dalibor Froncek$^{1,{2}}$}}\\
\normalsize $^1$AGH University of Krakow, Poland, cichacz@agh.edu.pl\\
\normalsize $^2$University of Minnesota Duluth, U.S.A.,  dalibor@d.umn.edu}
\maketitle

\begin{abstract}
Let $\Gamma$ be a group of order $n^2$ and $SMS_{\Gamma}(n)=(a_{i,j})_{n\times n}$ be an $n\times n$ array whose entries are all distinct elements of $\Gamma$. If there exists an element $\mu\in\Gamma$ such that for every row $i$, there exists an ordering of elements such that 
	$$
		a_{i,j_1} a_{i,j_2} \dots a_{i,j_{n-1}} a_{i,j_n} = \mu
	$$
	and for every column $j$ there exists an ordering of elements such that 
	$$
		a_{i_1,j} a_{i_2,j} \dots a_{i_{m-1},j} a_{i_m,j} = \mu,
	$$
	then $SMS_{\Gamma}(n)$ is called a \emph{$\Gamma$-semi-magic square of side $n$} and $\mu$ is called a \emph{magic constant}.

We provide a complete characterization of semi-magic squares of side $n$ whose entries belong to a dihedral group $D_k$. Moreover, we show that in our constructions a single semi-magic square may admit two distinct magic constants, depending on the order in which the products are computed.
\end{abstract}



%\newpage
\section{Introduction}\label{sec:intro}

Magic squares are one of the oldest mathematical structures, reportedly dating back to the 4-th century BCE. A \emph{magic square} of order $n$ is an $n\times n$ array with entries $1,2,\dots, n^2$ such that the sum of each row, column, and the main forward and backward diagonal is the same \emph{magic constant} $c=n(n^2+1)/2$. When we only require the row and column sum to be equal and disregard the diagonals, we speak about a \emph{semi-magic square.} For an overview, see, e.g., Chapter 6 Section 33 in~\cite{handbook}. A complete characterization is well known.

\begin{theorem}[\cite{handbook}] \label{thm:msquare}
	There exists a  {magic square} of side $n$ if and only if $n>2$.
\end{theorem}
A generalization of a magic square is an $m\times n$ \emph{magic rectangle} with entries $1,2,\dots,mn$ where all row sums are equal to the \emph{row constant} $r=n(mn+1)/2$ all column sums are equal to the \emph{column constant} $c=m(mn+1)/2$. A semi-magic square is then an $n\times n$ magic rectangle. Magic squares and rectangles can be generalized in many different ways. For instance, we may require that the entries are  elements of an Abelian group (see~\cite{CicJaco,CicFro,Cichacz-Hinc-2,Evans,Pellegrini,Sun-Yihui,ref_YuFengLiu}). 
%\edtS
In contrast, the case of non-Abelian groups has so far been considered only in~\cite{Froncek_dih_SMS_n=0mod4}. Namely, the second author recently tackled the problem of dihedral groups and constructed $D_{2n^2}$-semimagic rectangles of side $n$ over dihedral groups $D_{2n^2}$ of order $4n^2$ for any $n\equiv0\pmod4$~\cite{Froncek_dih_SMS_n=0mod4}. 
%\edtS
In this paper, we extend this result to all admissible values of $n$, using simpler construction methods. Furthermore, we show that in our constructions, a given semi-magic square may admit two distinct magic constants, depending on the order in which the products are taken.

\vskip6pt
\noindent
\textbf{Disclaimer} \textit{Some parts of this and the following section may be similar or identical to corresponding parts of paper~\cite{Froncek_dih_SMS_n=0mod4} by the second author on similar topic.}

%\newpage
\section{Definitions and necessary conditions}\label{sec:definitions}

It is common to refer to an $n\times n$ magic square as a square of \emph{order $n$}. To avoid confusion between the order of the group $\Gamma$ and the square, we reserve the word ``order'' to the group and speak about magic rectangle of \emph{side $n$}.
{We use the notation $[a,b]$ for the set $\{a,a+1,\dots,b\}$ of consecutive integers. When $a=0$, we may use just $[b]$.}

\begin{definition}\label{def:gamma-magic-rect}
	Let $\Gamma$ be a group of order $mn$ and $MR_{\Gamma}(m,n)= (a_{i,j})$ an $m\times n$ array whose entries are all elements of $\Gamma$. If for every row $i$ there exists an ordering of elements such that 
	\begin{equation*}%\label{key}
		a_{i,j_1} a_{i,j_2} \dots a_{i,j_{n-1}} a_{i,j_n} = \rho
	\end{equation*}
	and for every column $j$ there exists an ordering of elements such that 
	\begin{equation*}%\label{key}
		a_{i_1,j} a_{i_2,j} \dots a_{i_{m-1},j} a_{i_m,j} = \sigma,
	\end{equation*}
	then $MR_{\Gamma}(m,n)$ is called \emph{$\Gamma$-magic rectangle.}
%%%%%%%%%%%%%%%%%%%%%%%%%%%
	If for every row and column the ordering is
	\begin{equation*}%\label{key}
		a_{i,1} a_{i,2} \dots a_{i,{n-1}} a_{i,n} = \rho
	\end{equation*}
	and  
	\begin{equation*}%\label{key}
		a_{m,j} a_{m-1,j} \dots a_{{2},j} a_{1,j} = \sigma,
	\end{equation*}
	then $MR_{\Gamma}(m,n)$ is \emph{linearly $\Gamma$-magic.} 
	If for every row $i$ there exists $j$ (not necessarily the same for all $i$) such that
	\begin{equation*}%\label{key}
		a_{i,j+1} a_{i,j+2} 
		\dots
		a_{i,n} a_{i,1} 		
		\dots 
		a_{i,j-1} a_{i,j} = \rho
	\end{equation*}
	and for every column $j$ there exists $i$ such that
	\begin{equation*}%\label{key}
		a_{i-1,j} a_{i-2,j} 
		\dots
		a_{m,j} a_{1,j} 				
		\dots 
		a_{i+1,j} a_{i,j} = \sigma,
	\end{equation*}
	then we call the ordering \emph{circular} and $MR_{\Gamma}(m,n)$ \emph{circularly $\Gamma$-magic}. 
%%%%%%%%%%%%%%%%%%%%%%%%%%%
Finally, if $m=2s,n=2t$ and for every row $i$ there exists $j\in[1,t]$ and $j'\in[t+1,n]$ (not necessarily the same for all $i$) such that
\begin{equation*}%\label{key}
	a_{i,j+1} a_{i,j+2} 
	\dots
	a_{i,s} a_{i,1} 		
	\dots 
	a_{i,j-1} a_{i,j} = \rho_1
\end{equation*}
and
\begin{equation*}%\label{key}
	a_{i,j'+1} a_{i,'+2} 
	\dots
	a_{i,m} a_{i,s+1} 		
	\dots 
	a_{i,j'-1} a_{i,j'} = \rho_2
\end{equation*}
and for every column $j$ there exists $i\in[1,s]$ and $i'\in[s+1,m]$ such that
\begin{equation*}%\label{key}
	a_{i-1,j} a_{i-2,j} 
	\dots
	a_{1,j} a_{n,j} 				
	\dots 
	a_{i+1,j} a_{i,j} = \sigma_1,
\end{equation*}
and
\begin{equation*}%\label{key}
	a_{i'-1,j} a_{i'-2,j} 
	\dots
	a_{1,j} a_{n,j} 				
	\dots 
	a_{i'+1,j} a_{i',j} = \sigma_2,
\end{equation*}
then we call the ordering \emph{semi-circular} and $MR_{\Gamma}(m,n)$ \emph{semi-circularly $\Gamma$-magic}. 
%%%%%%%%%%%%%%%%%%%%%%%%%%%
\end{definition}

\begin{definition}\label{def:gamma-magic-square}
	Let $\Gamma$ be a group of order $n^2$ and $MR_{\Gamma}(n,n)$ a {$\Gamma$-magic rectangle.} If the row and column products are equal, that is, $\rho=\sigma$, then we call $MR_{\Gamma}(n,n)$ a \emph{$\Gamma$-semi-magic square} and if moreover the products of both the main and backward main diagonals are equal to $\rho=\sigma$, then $MR_{\Gamma}(n,n)$ is called a \emph{$\Gamma$-magic square}. 
	
	The notions of  \emph{linearly, circularly and semi-circularly $\Gamma$-magic} and \emph{$\Gamma$-semi-magic} squares are defined as in Definition~\ref{def:gamma-magic-rect}.
\end{definition}	
	
Observe that in the case where $\Gamma$ is an Abelian group, the order of elements in each product is irrelevant, as the group operation is commutative. Consequently, the conditions on the orderings in the definitions above are automatically satisfied.

The dihedral group $D_{k}$ of order $2k$ (sometimes also denoted by $D_{2k}$) is the group consisting of $k$ rotations $r_i$ and $k$ reflections $s_i$, where the rotations form a cyclic group of order $k$ and and each reflection generates a subgroup of order 2. More formal definition is below.

\begin{definition}\label{def:dihedral}
	The \emph{dihedral group}  $D_{k}$ of order $2k$ where $k\geq3$ is defined on the set of elements $\{r_0,r_1,\dots,r_{k-1},s_0,s_1,\dots,s_{k-1}\}$ where $r_0=e, r_1=r,  r_i=r^i$, $s_0=s$, $s_i=r^i s$, $s^2_i=e$ and $r^is=sr^{-i}$ for $i=0,1,\dots,k-1$.	The elements $r_i$ are called \emph{rotations}, and $s_i$ are called \emph{reflections}.
\end{definition}

An important property of $D_{k}$ will be used in our constructions. If follows directly from the definition.

\begin{prop}\label{prop:sr^is}
	In any dihedral group $D_{k}$, we have $sr^is=r^{-i}$ for every $i=0,1,\dots,k-1$.
\end{prop}

When we have a $\Gamma$-magic or semi-magic square where $\Gamma=D_k$, the dihedral group on $2k$ elements, we speak about a \emph{dihedral magic or semi-magic square}.

We begin by stating the obvious necessary conditions. Since every dihedral group is of even order, a magic square $MS_{D_k}$ of odd order cannot exist. The following necessary condition follows from straightforward observations made in~\cite{Froncek_dih_SMS_n=0mod4}.

\begin{theorem}[\cite{Froncek_dih_SMS_n=0mod4}]\label{thm:necessary}
	If a dihedral semi-magic square $SMS_{D_k}(n)$ exists, then both $n$ and $k$ must be even.
\end{theorem}

%\input{30_related_results_03a}
%\input{40_necessary_02a}

%\newpage
\section{Preliminaries}\label{sec:prel} 

By Theorem~\ref{thm:necessary}, the necessary condition for the existence of $SMS_{D_k}({n})$  is that $k$ and $n$ are even, thus $n=4m^2$ and $k=2m^2$ for some positive integer $m>2$.  	

Theorem~\ref{thm:msquare} implies the existence  a magic square $\tilde{M}(m)$ of side $m>2$ with the magic constant $\tilmu$. Using the  square $\tilde{M}(m)=(\tilde{m}_{i,j})_{m\times m}$ we build three \emph{power squares} of size $m\times m$: $E(m)$, $O(m)$ and $T(m)$ as follows. We will later use them to construct $SMS_{D_{2m^2}}(2m)$.

\begin{definition}\label{def:power squares}
	Let $\tilde{M}(m)$ be a magic square with entries $\tilde{m}_{i,j}$ and magic constant $\tilmu$. Define \emph{power squares} $E(m), O(m)$ and $T(m)$ with entries $e_{i,j}, o_{i,j}$ and $t_{i,j}$, respectively, for 
	$i,j\in[0,m-1]$
	as follows:
	\begin{itemize}
		\item $e_{i,j}=2\tilde{m}_{i,j} \pmod{2m^2}$
		\item $o_{i,j}=2\tilde{m}_{i,j}{+1} \pmod{2m^2}$
		\item $t_{i,j}=-2\tilde{m}_{i,j}+1 \pmod{2m^2}$ when $m$ is even
		\item $t_{i,j}=-2\tilde{m}_{i,j}+m-2 \pmod{2m^2}$ when $m$ is odd
	\end{itemize}
	We notice that in both cases $t_{i,j}$ is odd.
\end{definition}

	The sums in rows and columns of the squares are constant, as can be easily verified. We calculate them again modulo $2m^2$.
	
\begin{obs}\label{obs:power square sum}
	The row and column sums in the power squares $E(m), O(m),$ and $T(m)$ are
	\begin{itemize}
	\item $\tilde{e}=\sum_{j=0}^{m-1}e_{i,j}=\sum_{i=0}^{m-1}{2\tilde{m}}_{i,j}=2\tilmu\pmod{2m^2}$
	\item $\tilde{o}=\sum_{j=0}^{m-1}{o}_{i,j}=\sum_{i=0}^{m-1}{(2\tilde{m}_{i,j}+1)}=(2\tilmu+m)\pmod{2m^2}$
	%	\item $\tilde{t}=\sum_{j=0}^{m-1}e_{i,j}=\sum_{i=0}^{m-1}e_{i,j}=(-2\tilmu+mh)\pmod{2m^2}$
	\item $\tilde{t}=\sum_{j=0}^{m-1}{t}_{i,j}=\sum_{i=0}^{m-1}{(-2\tilde{m}_{i,j}+1)}=(-2\tilmu+m)\pmod{2m^2}$ for $m$ even
	\item $\tilde{t}=\sum_{j=0}^{m-1}{t}_{i,j}=\sum_{i=0}^{m-1}{(-2\tilde{m}_{i,j}+m-2)}=(-2\tilmu+m^2-2m)\pmod{2m^2}$ for $m$ odd
	\end{itemize}
\end{obs}	
% ending blue

%%%%%%%%%%%%%%%%%%%%%%%%%%%%%%%%%%%%%%%%%
%%%%%%%%%%%%%%%%%%%%%%%%%%%%%%%%%%%%%%%%%
%%%%%%%%%%%%%%%%%%%%%%%%%%%%%%%%%%%%%%%%%
%%%%%%%%%%%%%%%%%%%%%%%%%%%%%%%%%%%%%%%%%

\begin{exm}\label{exm:E(4)}
	In Figure~\ref{fig:E(4)} we show the construction of power squares $E(4)$, $O(4)$ and $T(4)$.
\end{exm}

\begin{figure}[H]
\begin{subfigure}{1\linewidth}
$$
\begin{array}{|c|c|c|c|}
\hline
16 & 2  & 3  & 13 \\ \hline
5  & 11 & 10 & 8  \\ \hline
9  & 7  & 6  & 12 \\ \hline
4  & 14 & 15 & 1  \\ \hline
\end{array}
$$
\caption{$\tilde{M}(4)$, $\tilde{\mu}=34$}
\end{subfigure}\\
\hfill
\begin{subfigure}{0.33\linewidth}
$$
\begin{array}{|c|c|c|c|}
\hline
0 & 4  & 6  & 26 \\ \hline
10  & 22 & 20 & 16  \\ \hline
18  & 14  & 12  & 24 \\ \hline
8  & 28 & 30 & 2  \\ \hline
\end{array}
$$
\caption{ $E(4)$, $\tilde{e}=4$}
\end{subfigure}
\begin{subfigure}{0.33\linewidth}
$$
\begin{array}{|c|c|c|c|c|c|}
\hline
1 & 5  & 7  & 27 \\ \hline
11  & 23 & 21 & 17  \\ \hline
19  & 15  & 13  & 25 \\ \hline
9  & 29 & 31 & 3  \\ \hline
\end{array}
$$
\caption{ $O(4)$, $\tilde{o}=8$}
\end{subfigure}
\hfill
\begin{subfigure}{0.32\linewidth}
$$
\begin{array}{|c|c|c|c|c|c|}
\hline
1 & 29 & 27 & 7 \\ \hline
23 & 11 & 13 & 17 \\ \hline
15 & 19 & 21 & 9 \\ \hline
25 & 5 & 3 & 31 \\ \hline
\end{array}
$$
\caption{$T(4)$, $\tilde{t}=0$}
\end{subfigure}

\caption{Example of $E(4)$, $O(4)$ and $T(4)$}\label{fig:E(4)}
\end{figure}

\begin{exm}\label{exm:E(5)}
	In Figure~\ref{fig:E(5)} we show the construction of power squares $E(5)$, $O(5)$ and $T(5)$.
\end{exm}

\begin{figure}[H]
\begin{subfigure}{1\linewidth}
$$
\begin{array}{|c|c|c|c|c|}
\hline
17 & 24 & 1  & 8  & 15 \\\hline
23 & 5  & 7  & 14 & 16 \\\hline
4  & 6  & 13 & 20 & 22 \\\hline
10 & 12 & 19 & 21 & 3  \\\hline
11 & 18 & 25 & 2  & 9 \\\hline
\end{array}
$$
\caption{$\tilde{M}(5)$, $\tilde{\mu}=65$}
\end{subfigure}\\
\hfill
\begin{subfigure}{0.33\linewidth}
$$
\begin{array}{|c|c|c|c|c|}
\hline
34 & 48 & 2  & 16 & 30 \\\hline
46 & 10 & 14 & 28 & 32 \\\hline
8  & 12 & 26 & 40 & 44 \\\hline
20 & 24 & 38 & 42 & 6  \\\hline
22 & 36 & 0  & 4  & 18\\ \hline
\end{array}
$$
\caption{ $E(5)$, $\tilde{e}=30$}
\end{subfigure}
\begin{subfigure}{0.33\linewidth}
$$
\begin{array}{|c|c|c|c|c|c|}
\hline
35 & 49 & 3  & 17 & 31 \\\hline
47 & 11 & 15 & 29 & 33 \\\hline
9  & 13 & 27 & 41 & 45 \\\hline
21 & 25 & 39 & 43 & 7  \\\hline
23 & 37 & 1  & 5  & 19\\ \hline
\end{array}
$$
\caption{ $O(5)$, $\tilde{o}={35}$}
\end{subfigure}
\hfill
\begin{subfigure}{0.32\linewidth}
$$
\begin{array}{|c|c|c|c|c|c|}
\hline
19 & 5  & 1  & 37 & 23 \\ \hline
7  & 43 & 39 & 25 & 21 \\ \hline
45 & 41 & 27 & 13 & 9  \\ \hline
33 & 29 & 15 & 11 & 47 \\ \hline
31 & 17 & 3  & 49 & 35 \\ \hline
\end{array}
$$
\caption{$T(5)$, $\tilde{t}=35$}
\end{subfigure}

\caption{Example of $E(5)$, $O(5)$ and $T(5)$}\label{fig:E(5)}
\end{figure}

\newpage
Rather than using the group elements and entries and performing the group operation, we simply use the exponents in elements $r_i=r^i$ and $s_i=r^is$. Then, when in our square we would perform the product (read from right to left) $\dots r^i r^j\dots$, we instead just add the exponents $\dots i+j\dots$. For reflections, we recall that from Proposition~\ref{prop:sr^is} it follows that 
$$
	(r^is)(r^js) = r^i(s r^j s) = r^i r^{-j}
$$
and we use just the term $i-j$.

We  will build $MS_{\Gamma}({n})$ from partial squares $Q^{uv}({m})=(q^{uv}_{i,j})_{m\times m}$ for $u,v\in\{1,2\}$. Because the partial squares are always identified by the superscript $uv$, we believe that no confusion will arise when we drop the square sides and write simply  $Q^{uv}$ instead of $Q^{uv}({m})$.

The row sums in row $i$ will be  $\rho^{uv}_{i}$ in $Q^{uv}$ and it is calculated according to the order: 
 $$\rho^{uv}_{i}=q^{uv}_{i,i}q^{uv}_{i,i+1}\ldots q^{uv}_{i,m-1}q^{uv}_{i,0}q^{uv}_{i,1}\ldots q^{uv}_{i,i-1},$$
where the subscripts are taken modulo $m$.

{In other words, we always start with the entry just before the square diagonal and proceed to the left cyclically until we stop at the diagonal.}
 
The column sums in column $j$ will be  $\sigma^{uv}_{j}$ 
and it is calculated according to the order: 
 $$\sigma^{uv}_{j}=q^{uv}_{j,j}q^{uv}_{j-1,j}\ldots q^{uv}_{0,j} q^{uv}_{m-1,j}q^{uv}_{{m-2},j}\ldots q^{uv}_{j+1,j}.$$

Similarly as in the row product, we always start with the entry below the square diagonal and proceed down cyclically until we arrive at the diagonal.

%%%%%%%%%%%%%%%%%%%%%%%%%%%%%%%%%%%%%%%%%%%%%%%%%%%%%%%%%%%%%%%%%
%%%%%%%%%%%%%%%%%%%%%%%%%%%%%%%%%%%%%%%%%%%%%%%%%%%%%%%%%%%%%%%%%
%%%%%%%%%%%%%%%%%%%%%%%%%%%%%%%%%%%%%%%%%%%%%%%%%%%%%%%%%%%%%%%%%
%%%%%%%%%%%%%%%%%%%%%%%%%%%%%%%%%%%%%%%%%%%%%%%%%%%%%%%%%%%%%%%%%
%%%%%%%%%%%%%%%%%%%%%%%%%%%%%%%%%%%%%%%%%%%%%%%%%%%%%%%%%%%%%%%%%
%%%%%%%%%%%%%%%%%%%%%%%%%%%%%%%%%%%%%%%%%%%%%%%%%%%%%%%%%%%%%%%%%
%%%%%%%%%%%%%%%%%%%%%%%%%%%%%%%%%%%%%%%%%%%%%%%%%%%%%%%%%%%%%%%%%
%%%%%%%%%%%%%%%%%%%%%%%%%%%%%%%%%%%%%%%%%%%%%%%%%%%%%%%%%%%%%%%%%
%%%%%%%%%%%%%%%%%%%%%%%%%%%%%%%%%%%%%%%%%%%%%%%%%%%%%%%%%%%%%%%%%

%\newpage
\section{Construction for $m$ even}\label{sec:even}

We present a general construction for $k\geq8$, using the partial squares $Q^{uv}$. {The squares $Q^{11}$ and $Q^{22}$ will consist of rotations and $Q^{12}$ and $Q^{21}$ of reflections.}

 %The column sums in column $j$ will be  $\sigma^{uv}_{j}$ and it is calculated according to the order: 
% $$\sigma^{uv}_{j}=q^{uv}_{j,j}q^{uv}_{j+1,j}\ldots q^{uv}_{m-1,j}q^{uv}_{0,j}q^{uv}_{1,j}\ldots q^{uv}_{j-1,j}.$$

\begin{const}\label{const:mevenrot}\emph{Rotations}
	
	We construct first the squares $Q^{11}$ and $Q^{22}$. Let  $E(m)=(e_{i,j})_{m\times m}$ and $O(m)=(o_{i,j})_{m\times m}$ be the power squares defined in Section~\ref{sec:prel}. 	Let 
    $$q^{11}_{i,j}=\begin{cases}r^{e_{i,j}},& i+j \;\mathrm{even},\\
    r^{o_{i,j}},& i+j \;\mathrm{odd},\\\end{cases}\;\;\;\;q^{22}_{i,j}=\begin{cases}r^{e_{i,j+1}},& i+j \;\mathrm{even},\\
    r^{o_{i,j+1}},& i+j \;\mathrm{odd}.\\\end{cases}$$
     Observe that for   any $i,j\in[0,m-1]$ we have 
\begin{align*}
	{e_{i,l}}+{o_{i,l+1}}
		&={2\tilde{m}_{i,l}}+{2\tilde{m}_{i,l+1}}+1,\\
	{o_{l-1,j}}+e_{l,j}
		&={2\tilde{m}_{l-1,j}}+1+2\tilde{m}_{l,j},
\end{align*} 
where the subscripts are taken modulo $m$.  Therefore 
 \begin{alignat*}{2}
	\rho^{11}_{i}
		&=r^{e_{i,i}}r^{o_{i,i+1}}r^{e_{i,i+2}}r^{o_{i,i+3}}\ldots r^{e_{i,i-2}}r^{o_{i,i-1}}
		&&=r^{2\tilde{\mu}+m/2},\\
\sigma^{11}_{j}
		&=r^{e_{j,j}}r^{o_{j-1,j}}r^{e_{j-2,j}}r^{o_{j-3,j}}\ldots r^{e_{j+2,j}}r^{o_{j+1,j}}
		&&=r^{2\tilde{\mu}+m/2},\\
\rho^{22}_{i}
		&=r^{e_{i,i+1}}r^{o_{i,i+2}}r^{e_{i,i+3}}r^{o_{i,i+4}}\ldots r^{e_{i,i-1}}r^{o_{i,i}}
		&&=r^{2\tilde{\mu}+m/2},\\
\sigma^{22}_{j}
		&=r^{e_{j-1,j}}r^{o_{j-2,j}}r^{e_{j-3,j}}r^{o_{j-4,j}}\ldots r^{e_{j+1,j}}r^{o_{j,j}}
		&&=r^{2\tilde{\mu}+m/2}.
\end{alignat*}
\end{const}

We show an example for $m=4$ in Figure~\ref{fig:Q^ss(4)}.

\begin{figure}[H]
\begin{center}

\begin{subfigure}{1\linewidth}
$$
\begin{array}{|c|c|c|c|}
\hline
r^{e_{1,1}} 	& r^{o_{1,2}} 	& r^{e_{1,3}}  	& r^{o_{1,4}} \\ \hline
r^{o_{2,1}} 	& r^{e_{2,2}} 	& r^{o_{2,3}}  	& r^{e_{2,4}} \\ \hline
r^{e_{3,1}} 	& r^{o_{3,2}} 	& r^{e_{3,3}}  	& r^{o_{3,4}} \\ \hline
r^{o_{4,1}} 	& r^{e_{4,2}} 	& r^{o_{4,3}}  	& r^{e_{4,4}} \\ \hline
\end{array}=
\begin{array}{|c|c|c|c|}
\hline
r^{0} 	& r^{5}  	& r^{6}  	& r^{27} \\ \hline
r^{11} 	& r^{22}  	& r^{21}  	& r^{16} \\ \hline
r^{18} 	& r^{15}  	& r^{12}  	& r^{25} \\ \hline
r^{9} 	& r^{28}  	& r^{31}  	& r^{2} \\ \hline
\end{array}
$$
\caption{$Q^{11}(4)$}
\end{subfigure}\\
\begin{subfigure}{1\linewidth}
$$
\begin{array}{|c|c|c|c|c|c|}
\hline
r^{e_{1,2}} 	& r^{o_{1,3}} 	& r^{e_{1,4}}  	& r^{o_{1,1}} \\ \hline
r^{o_{2,2}} 	& r^{e_{2,3}} 	& r^{o_{2,4}}  	& r^{e_{2,1}} \\ \hline
r^{e_{3,2}} 	& r^{o_{3,3}} 	& r^{e_{3,4}}  	& r^{o_{3,1}} \\ \hline
r^{o_{4,2}} 	& r^{e_{4,3}} 	& r^{o_{4,4}}  	& r^{e_{4,1}} \\ \hline
\end{array}=
\begin{array}{|c|c|c|c|c|c|}
\hline
r^{4} 	& r^{7}  	& r^{26}  	& r^{1} \\ \hline
r^{23} 	& r^{20}  	& r^{17}  	& r^{10} \\ \hline
r^{14} 	& r^{13}  	& r^{24}  	& r^{19} \\ \hline
r^{29} 	& r^{30}  	& r^{3}  	& r^{8} \\ \hline
\end{array}
$$
\caption{$Q^{22}(4)$}
\end{subfigure}

\end{center}

\caption{Squares $Q^{11}$ and $Q^{22}$}\label{fig:Q^ss(4)}
\end{figure}   
   
\begin{const}\label{const:mevenref}\emph{Reflections}

We construct now the squares $Q^{12}$ and $Q^{21}$. Let  $E(m)=(e_{i,j})_{m\times m}$, $O(m)=(o_{i,j})_{m\times m}$ and $T(m)=(t_{i,j})_{m\times m}$ be the power squares defined in Section~\ref{sec:prel}. 
	Let 
    $$q^{12}_{i,j}=\begin{cases}r^{e_{i,j}}s,& i+j \;\mathrm{even},\\
    r^{t_{i,j}}s,& i+j \;\mathrm{odd},\\\end{cases}\;\;\;\;q^{21}_{i,j}=\begin{cases}r^{e_{i,j+1}}s,& i+j \;\mathrm{even},\\
    r^{t_{i,j+1}}s,& i+j \;\mathrm{odd}.\\\end{cases}$$
   
    Observe that for   any $i,j,l\in[0,m-1]$ we have 
\begin{align*}
	{e_{i,l}}-{t_{i,l+1}}
		&={2\tilde{m}_{i,l}}-(-{2\tilde{m}_{i,l+1}}+1)\\
		&={2\tilde{m}_{i,l}}+{2\tilde{m}_{i,l+1}}-1,\\
	{e_{l,j}}-t_{l-1,j}
		&={2\tilde{m}_{l,j}}-(-2\tilde{m}_{l-1,j}+1)\\
		&={2\tilde{m}_{l,j}}+2\tilde{m}_{l-1,j}-1,
\end{align*} 
where the subscripts are taken modulo $m$. Therefore, 
\begin{align*}
	\rho^{12}_{i}
		&=r^{e_{i,i}}sr^{t_{i,i+1}}sr^{e_{i,i+2}}sr^{t_{i,i+3}}s\ldots r^{e_{i,i-2}}sr^{t_{i,i-1}}s\\
		&=r^{e_{i,i}}r^{-t_{i,i+1}}r^{e_{i,i+2}}r^{-t_{i,i+3}}s\ldots r^{e_{i,i-2}}r^{-t_{i,i-1}}s\\
		&=r^{2\tilde{\mu}-m/2},\\
	\sigma^{12}_{j}
		&=r^{e_{j,j}}sr^{t_{j-1,j}}sr^{e_{j-2,j}}sr^{t_{j-3,j}}s\ldots r^{e_{j+2,j}}sr^{t_{j+1,j}}s\\
		&=r^{e_{j,j}}r^{-t_{j-1,j}}r^{e_{j-2,j}}r^{-t_{j-3,j}}\ldots r^{e_{j+2,j}}r^{t_{j+1,j}}\\
		&=r^{2\tilde{\mu}-m/2},\\
	\rho^{21}_{i}
		&=r^{e_{i,i+1}}sr^{t_{i,i+2}}sr^{e_{i,i+3}}sr^{t_{i,i+4}}s\ldots r^{e_{i,i-1}}sr^{t_{i,i}}s\\
		&=r^{e_{i,i+1}}r^{-t_{i,i+2}}r^{e_{i,i+3}}r^{-t_{i,i+4}}\ldots r^{e_{i,i-1}}r^{-t_{i,i}}s\\
		&=r^{2\tilde{\mu}-m/2},\\
	\sigma^{21}_{j}
		&=r^{e_{j-1,j}}sr^{t_{j-2,j}}sr^{e_{j-3,j}}sr^{t_{j-4,j}}s\ldots r^{e_{j+1,j}}sr^{t_{j,j}}s\\
		&=r^{e_{j-1,j}}r^{-t_{j-2,j}}r^{e_{j-3,j}}r^{-t_{j-4,j}}\ldots r^{e_{j+1,j}}r^{-t_{j,j}}\\
		&=r^{2\tilde{\mu}-m/2}.
\end{align*} 
\end{const}

We show an example for $m=4$ in Figure~\ref{fig:Q^st(4)}.

\begin{figure}[H]
\begin{center}
\begin{subfigure}{1\linewidth}
$$
\begin{array}{|c|c|c|c|}
\hline
r^{e_{1,1}}s 	& r^{t_{1,2}}s 	& r^{e_{1,3}}s  	& r^{t_{1,4}}s \\ \hline
r^{t_{2,1}}s 	& r^{e_{2,2}}s 	& r^{t_{2,3}}s  	& r^{e_{2,4}}s \\ \hline
r^{e_{3,1}}s 	& r^{t_{3,2}}s 	& r^{e_{3,3}}s  	& r^{t_{3,4}}s \\ \hline
r^{t_{4,1}}s 	& r^{e_{4,2}}s 	& r^{t_{4,3}}s  	& r^{e_{4,4}}s \\ \hline
\end{array}
=\begin{array}{|c|c|c|c|}
\hline
r^{0}s 		& r^{29}s  	& r^{6}s  	& r^{7}s \\ \hline
r^{23}s 	& r^{22}s  	& r^{13}s  	& r^{16}s \\ \hline
r^{18}s 	& r^{19}s  	& r^{12}s  	& r^{9}s \\ \hline
r^{25}s 	& r^{28}s  	& r^{3}s  	& r^{2}s \\ \hline
\end{array}
$$
\caption{$Q^{12}(4)$}
\end{subfigure}\\
\begin{subfigure}{1\linewidth}
$$
\begin{array}{|c|c|c|c|c|c|}
\hline
r^{e_{1,2}}s 	& r^{t_{1,3}}s 	& r^{e_{1,4}}s  	& r^{t_{1,1}}s \\ \hline
r^{t_{2,2}}s 	& r^{e_{2,3}}s 	& r^{t_{2,4}}s  	& r^{e_{2,1}}s \\ \hline
r^{e_{3,2}}s 	& r^{t_{3,3}}s 	& r^{e_{3,4}}s  	& r^{t_{3,1}}s \\ \hline
r^{t_{4,2}}s 	& r^{e_{4,3}}s 	& r^{t_{4,4}}s  	& r^{e_{4,1}}s \\ \hline
\end{array}
=\begin{array}{|c|c|c|c|c|c|}
\hline
r^{4}s 		& r^{27}s  	& r^{26}s  	& r^{1}s \\ \hline
r^{11}s 	& r^{20}s  	& r^{17}s  	& r^{10}s \\ \hline
r^{14}s 	& r^{21}s  	& r^{24}s  	& r^{15}s \\ \hline
r^{5}s 		& r^{30}s  	& r^{31}s  	& r^{8}s \\ \hline
\end{array}
$$
\caption{$Q^{21}(4)$}
\end{subfigure}

\end{center}

\caption{Squares $Q^{12}$ and $Q^{21}$}\label{fig:Q^st(4)}
\end{figure}  

We are now ready to state our first result.

\begin{theorem}\label{thm:n=0mod4}\label{thm:meven}
	There exists a semi-circularly $D_{2m^2}$-semi-magic square $Q(2m)$ for every even $m$, $m\geq4$. 
\end{theorem}

\begin{proof}
Let $Q^{u,v}$ be the four squares obtained by Constructions~\ref{const:mevenrot} and \ref{const:mevenref}
We will glue them to obtain a square $Q$ as in Figure~\ref{fig:Q}.

 \begin{figure}[H]%{0.5\textwidth}
 \begin{center}
 \begin{tabular}{|c|c|}
 \hline
     $Q^{11}$  & $Q^{12}$ \\ \hline
$Q^{21}$  & $Q^{22}$ \\ \hline

   \end{tabular}
 \caption{A square $Q=SMS_{D_{2m^2}}(2m)$}\label{fig:Q}
 \end{center}
 \end{figure}
	
 	Each row product is now performed as
\begin{align*}
 	\rho_i	
 		&=	(q^{u1}_{i,i}\ q^{u1}_{i,i+1}\dots  q^{u1}_{i,m-1} q^{u1}_{i,0} q^{u1}_{i,1}\dots q^{u1}_{i,i-1})\ 
 					(q^{u2}_{i,i}\ q^{u2}_{i,i+1}\dots q^{u2}_{i,m-1}  q^{u2}_{i,0} q^{u2}_{i,1}\dots  q^{u2}_{i,i-1})\\
 		&=\rho^{u1}_i \rho^{u2}_i\\ &=r^{2\tilde{\mu}+m/2}r^{2\tilde{\mu}-m/2}\\ &=r^{4\tilde{\mu}}
\end{align*}
and the column products as
\begin{align*}
 	\sigma_j
 		&=(q^{1v}_{j,j}\ q^{1v}_{j-1,j} \dots q^{1v}_{0,j}q^{1v}_{m-1,j}q^{1v}_{m-2,j} \dots q^{1v}_{j+1,j})\ 
 					(q^{2v}_{j,j}\ q^{2v}_{j-1,j} \dots q^{2v}_{0,j}q^{2v}_{m-1,j}q^{2v}_{m-2,j} \dots q^{2v}_{j+1,j})\\
		&=\sigma^{1v}_i \sigma^{2v}_i\\ &=r^{2\tilde{\mu}+m/2}r^{2\tilde{\mu}-m/2}\\ &=r^{4\tilde{\mu}}, 	  
\end{align*}
	where the subscripts are taken modulo $m$.
	All row and column products in the square $Q(2m)$ are semi-circular and equal to the magic constant $\mu=r^{4\tilde{\mu}}$, which completes the proof.
\end{proof}

\begin{exm}\label{exm:D(8)} In Figure~\ref{fig:D(8)} we show the construction of magic square $M_{D_{32}}(8)$ using power squares from Figure~\ref{fig:E(4)}.
\end{exm}

\begin{figure}

$$
\begin{array}{||c|c|c|c||c|c|c|c||}
\hline\hline
r^0 & r^5  & r^6  & r^{27} & r^0s & r^{29}s  & r^6s  & r^{7}s\\ \hline
r^{11}  & r^{22} & r^{21} & r^{16}& r^{23}s  & r^{22}s & r^{13}s & r^{16}s \\ \hline
r^{18}  &r^{ 15}  &r^{ 12}  & r^{25}&r^{18}s  &r^{ 19}s  &r^{ 12}s  & r^{9}s \\ \hline
r^{9}  & r^{28} & r^{31} & r^{2} &r^{25}s  & r^{28}s & r^{3}s & r^{2}s \\ \hline\hline
r^{4}s  & r^{27}s & r^{26}s & r^{1}s&r^{4}  & r^{7} & r^{26} & r^{1} \\ \hline
r^{11}s  &r^{ 20}s  &r^{ 17}s  & r^{10}s&r^{23}  &r^{ 20}  &r^{ 17}  & r^{10} \\ \hline
r^{14}s  & r^{21}s & r^{24}s & r^{15}s &r^{14}  & r^{13} & r^{24} & r^{19} \\ \hline
r^{5}s  &r^{ 30}s  &r^{31}s  & r^{8}s&r^{29}  &r^{ 30}  &r^{ 3}  & r^{8} \\ \hline\hline
\end{array}
$$
\caption{$SMS_{D_{32}}(8)$ with the magic constant $\mu=r^{8}$}\label{fig:D(8)}

\end{figure}

We now show that in our constructions, a given semi-magic square may admit two distinct magic constants, depending on the order in which the products are taken.

Let $c\in[0,m-1]$ and define the row and column products by
\begin{align*}
	\rho^{uv}_{i,c}
		&=q^{uv}_{i,i+c}\,q^{uv}_{i,i+c+1}\ldots q^{uv}_{i,i+c-1},\\
\sigma^{uv}_{j,c}
		&=q^{uv}_{j-c,j}\,q^{uv}_{j-c-1,j}\ldots q^{uv}_{j-c+1,j},
\end{align*}
where all subscripts are taken modulo $m$.

Thus $\rho^{uv}_{i}=\rho^{uv}_{i,0}$ and $\sigma^{uv}_{j}=\sigma^{uv}_{j,0}$ in $Q^{uv}$.

Observe that in Construction~\ref{const:mevenref} (that is, for reflection squares $Q^{uv}$, i.e. $u+v$ odd) we have
\[
\rho^{uv}_{i,c}
=
\begin{cases}
r^{2\tilde{\mu}-m/2}, & \text{if } c \text{ is even},\\[4pt]
r^{-2\tilde{\mu}+m/2}, & \text{if } c \text{ is odd},
\end{cases}
\qquad
\sigma^{uv}_{j,c}
=
\begin{cases}
r^{2\tilde{\mu}-m/2}, & \text{if } c \text{ is even},\\[4pt]
r^{-2\tilde{\mu}+m/2}, & \text{if } c \text{ is odd}.
\end{cases}
\]

Consequently, in the proof of Theorem~\ref{thm:n=0mod4}, the resulting magic constant depends on the chosen ordering. Indeed,
\[
\rho_i
=
\rho^{u1}_{i,0}\rho^{u2}_{i,c}
=
\begin{cases}
r^{4\tilde{\mu}}, & \text{if } c \text{ is even},\\[4pt]
r^{m}, & \text{if } c \text{ is odd},
\end{cases}
\]
and similarly for the column products,
\[
\sigma_j
=
\sigma^{1v}_{j,0}\sigma^{2v}_{j,c}
=
\begin{cases}
r^{4\tilde{\mu}}, & \text{if } c \text{ is even},\\[4pt]
r^{m}, & \text{if } c \text{ is odd}.
\end{cases}
\]
Thus, two distinct magic constants may arise from the same semi-magic square, depending on the order in which the products are taken.

For example, in Figure~\ref{fig:D(8)} we have
 $$\rho^{u1}_{i,0}\rho^{u2}_{i,1}
=\sigma^{1v}_{j,0}\sigma^{1v}_{j,1}=r^{4}$$ which is also a magic constant.

%\newpage
\section{Construction for $m$ odd}\label{sec:odd} 

We now present a general construction for $m$ odd. Let  $E(m)=(e_{i,j})_{m\times m}$, $O(m)=(o_{i,j})_{m\times m}$ and $T(m)=(t_{i,j})_{m\times m}$ be the power squares defined in Section~\ref{sec:prel}.

\begin{const}\label{const:Q11}\emph{Square $Q^{11}$}
	
	 Let  $E(m)=(e_{i,j})_{m\times m}$  be the power square defined in Section~\ref{sec:prel}. 	Let 
    $$q^{11}_{i,j}=r^{e_{i,j}},$$ for   any $i,j\in[0,m-1]$.  Therefore 
\begin{align*}
	\rho^{11}_{i}
		&=r^{e_{i,i}}r^{e_{i,i+1}}\ldots r^{e_{i,m-1}}r^{e_{i,0}}r^{e_{i,1}}\ldots r^{e_{i,i-1}}\\
		&=r^{2\tilde{\mu}},\\
	\sigma^{11}_{j}
	&=r^{e_{j,j}}r^{e_{j-1,j}}\ldots r^{e_{0,j}}r^{e_{m-1,j}}r^{e_{m-2,j}}\ldots r^{e_{j+1,j}}\\
	&=r^{2\tilde{\mu}}.
\end{align*}   
   \end{const} 

An example is shown in Figure~\ref{fig:Q^11(5)}.   
   
\begin{figure}[H]
\begin{center}
%\begin{subfigure}{1\linewidth}
$$
\begin{array}{|c|c|c|c|c|}
\hline
r^{e_{1,1}} & r^{e_{1,2}} & r^{e_{1,3}}  & r^{e_{1,4}}	& r^{e_{1,5}}\\ \hline
r^{e_{2,1}} & r^{e_{2,2}} & r^{e_{2,3}}  & r^{e_{2,4}}	& r^{e_{2,5}}\\ \hline
r^{e_{3,1}} & r^{e_{3,2}} & r^{e_{3,3}}  & r^{e_{3,4}}	& r^{e_{3,5}}\\ \hline
r^{e_{4,1}} & r^{e_{4,2}} & r^{e_{4,3}}  & r^{e_{4,4}}	& r^{e_{4,5}}\\ \hline
r^{e_{5,1}} & r^{e_{5,2}} & r^{e_{5,3}}  & r^{e_{5,4}}	& r^{e_{5,5}}\\ \hline
\end{array}=
\begin{array}{|c|c|c|c|c|}
\hline
r^{34} & r^{48} & r^{2}  & r^{16} & r^{30} 	\\ \hline
r^{46} & r^{10} & r^{14} & r^{28} & r^{32} 	\\ \hline
r^{8}  & r^{12} & r^{26} & r^{40} & r^{44} 	\\ \hline
r^{20} & r^{24} & r^{38} & r^{42} & r^{6} 	\\ \hline
r^{22} & r^{36} & r^{0}  & r^{4}  & r^{18} 	\\ \hline
\end{array}
$$
%\caption{$Q^{11}(5)$}
%\end{subfigure}
\end{center}
\caption{Square $Q^{11}(5)$}\label{fig:Q^11(5)}
\end{figure}

\begin{const}\label{const:Q22}\emph{Square $Q^{22}$}
	Let 
    $$q^{22}_{i,j}=\begin{cases}r^{e_{i,j}}s,&\mathrm{for}\;\;j=i,\\
    r^{t_{i,j}}s,&\mathrm{for}\;\; j=i+1,\\
    r^{o_{i,j}},&  \mathrm{otherwise}.\\\end{cases}$$

    Observe that for any $i\in[0,m-1]$ we have 
\begin{align*}
	{e_{i,i}}-{t_{i,i+1}}+\sum_{s=2}^{m-1}o_{i,i+s}
	&={2\tilde{m}_{i,i}}-(-{2\tilde{m}_{i,i+1}}+m-2)+\sum_{s=2}^{m-1}(2\tilde{m}_{i,i+s}+1)\\
	&=-(m-2)+(m-2)+\sum_{s=0}^{m-1}2\tilde{m}_{i,i+s}\\
	&=2\tilde{\mu},\end{align*} 
where the subscripts are taken modulo $m$. 
Moreover,  for   any $j\in[0,m-1]$ we have \begin{align*}
	{e_{j,j}}-{t_{j-1,j}}+\sum_{s=2}^{m-1}o_{j-s,j}
	&={2\tilde{m}_{j,j}}-(-{2\tilde{m}_{j-1,j}}+m-2)+\sum_{s=2}^{m-1}(2\tilde{m}_{j-s,j}+1)\\
	&=-(m-2)+(m-2)+\sum_{s=0}^{m-1}2\tilde{m}_{j-s,j}\\
 	&=2\tilde{\mu},
\end{align*} 
where the subscripts are taken modulo $m$. Therefore, 
\begin{align*}
\rho^{22}_{i}
	&=r^{e_{i,i}}sr^{t_{i,i+1}}sr^{o_{i,i+2}}r^{o_{i,i+2}}\ldots r^{o_{i,m-1}}r^{o_{i,0}}r^{o_{i,1}}\ldots r^{o_{i,i-1}}\\
	&=r^{e_{i,i}}r^{-t_{i,i+1}}r^{o_{i,i+2}}r^{o_{i,i+2}}\ldots r^{o_{i,m-1}}r^{o_{i,0}}r^{o_{i,1}}\ldots r^{o_{i,i-1}}\\
	&=r^{2\tilde{\mu}},\\
\sigma^{22}_{j}
	&=r^{e_{j,j}}sr^{t_{j-1,j}}sr^{o_{j-2,j}}r^{o_{j-3,j}}\ldots r^{o_{0,j}}r^{o_{m-1,j}}r^{o_{m-2,j}}\ldots r^{o_{j+1,j}}\\
	&=r^{e_{j,j}}r^{-t_{j-1,j}}r^{o_{j-2,j}}r^{o_{j-3,j}}\ldots r^{o_{0,j}}r^{o_{m-1,j}}r^{o_{m-2,j}}\ldots r^{o_{j+1,j}}\\
	&=r^{2\tilde{\mu}}.
\end{align*}   
\end{const}

An example is shown in Figure~\ref{fig:Q^22(5)}.   
   
\begin{figure}[H]
\begin{center}
%\begin{subfigure}{1\linewidth}
$$
\begin{array}{|c|c|c|c|c|}
\hline
r^{e_{1,1}}s & r^{t_{1,2}}s & r^{o_{1,3}}  & r^{o_{1,4}}	& r^{o_{1,5}}\\ \hline
r^{o_{2,1}} & r^{e_{2,2}}s & r^{t_{2,3}}s  & r^{o_{2,4}}	& r^{o_{2,5}}\\ \hline
r^{o_{3,1}} & r^{o_{3,2}} & r^{e_{3,3}}s  & r^{t_{3,4}}s	& r^{o_{3,5}}\\ \hline
r^{o_{4,1}} & r^{o_{4,2}} & r^{o_{4,3}}  & r^{e_{4,4}}s	& r^{t_{4,5}}s\\ \hline
r^{t_{5,1}}s & r^{o_{5,2}}s & r^{o_{5,3}}  & r^{o_{5,4}}	& r^{e_{5,5}}s\\ \hline
\end{array}=
\begin{array}{|c|c|c|c|c|}
\hline
r^{34}s &  r^{5}s& r^3  & r^{17} & r^{31} \\\hline
r^{47} &r^{10}s &r^{39}s  & r^{29} & r^{33} \\\hline
r^{9}  & r^{13} & r^{26}s &r^{13}s  & r^{45} \\\hline
r^{21} & r^{25} & r^{39} & r^{42}s & r^{47}s \\\hline
r^{31}s & r^{37} & r^{1}  & r^{5}  & r^{18}s\\ \hline
\end{array}
$$
%\caption{$Q^{22}(5)$}
%\end{subfigure}
\end{center}
\caption{Square $Q^{22}(5)$}\label{fig:Q^22(5)}
\end{figure}

\begin{const}\label{const:Q12Q21}\emph{Square $Q^{12}$ and $Q^{21}$}

We construct now the squares $Q^{12}$ and $Q^{21}$. Let 
$$
	q^{12}_{i,j}=
\begin{cases}r^{o_{i,j}},& j=i,\\
	r^{e_{i,j}}s,& i+j\;\textrm{odd},\;\;j\in[i+1,m-1],\\
	r^{t_{i,j}}s,&  i+j\;\mathrm{even},\;j\in[i+2,m-1],\\
	r^{e_{i,j}}s,& i+j\;\textrm{even},\;j\in[0,i-1],\\
	r^{t_{i,j}}s,&  i+j\;\mathrm{odd},\;\;j\in[0,i-1].\\
\end{cases}
$$

An example is shown in Figure~\ref{fig:Q^12(5)}.   
   
\begin{figure}[H]
\begin{center}
%\begin{subfigure}{1\linewidth}
$$
\begin{array}{|c|c|c|c|c|}
\hline
r^{o_{1,1}} & r^{e_{1,2}}s & r^{t_{1,3}}s  & r^{e_{1,4}}s	& r^{t_{1,5}}s\\ \hline
r^{t_{2,1}}s & r^{o_{2,2}} & r^{e_{2,3}}s  & r^{t_{2,4}}s	& r^{e_{2,5}}s\\ \hline
r^{e_{3,1}}s & r^{t_{3,2}}s & r^{o_{3,3}}  & r^{e_{3,4}}s	& r^{t_{3,5}}s\\ \hline
r^{t_{4,1}}s & r^{e_{4,2}}s & r^{t_{4,3}}s  & r^{o_{4,4}}	& r^{e_{4,5}}s\\ \hline
r^{e_{5,1}}s & r^{t_{5,2}}s & r^{e_{5,3}}s  & r^{t_{5,4}}s	& r^{o_{5,5}} \\ \hline
\end{array}=
\begin{array}{|c|c|c|c|c|}
\hline
r^{35}&r^{48}s&r^{1}s&r^{16}s&r^{23}s	\\\hline
r^{7}s&r^{11}&r^{14}s&r^{25}s&r^{32}s	\\\hline
r^{8}s&r^{41}s &r^{27}&r^{40}s&r^{9}s 	\\\hline
r^{33}s&r^{24}s&r^{15}s&r^{43}&r^{6}s 	\\\hline
r^{22}s&r^{17}s&r^{0}s&r^{49}s&r^{19}	\\\hline
\end{array}
$$
%\caption{$Q^{12}(5)$}
%\end{subfigure}
\end{center}
\caption{Square $Q^{12}(5)$}\label{fig:Q^12(5)}
\end{figure}  
        
$$
	q^{21}_{i,j}=
	\begin{cases}
	r^{o_{i,j+1}},
		& j=i,\\
    r^{e_{i,j+1}}s,
    	& i+j\;\textrm{odd},\; j\in[i+2,m-1],\\
    r^{t_{i,j+1}}s,
    	&  i+j\;\textrm{even}, j\in[i+1,m-1],\\
    r^{e_{i,j+1}}s,
    & i+j\;\textrm{even},j\in[0,i-1],\\
    r^{t_{i,j+1}}s,
    & i+j\;\mathrm{odd},\;j\in[0,i-1].\\
	\end{cases}
$$
    Observe that for any $i,j,l\in[0,m-1]$, $l\neq i$ and $l\neq j $ we have 
\begin{align*}
{e_{i,l}}-{t_{i,l+1}}
	&={2\tilde{m}_{i,l}}-(-{2\tilde{m}_{i,l+1}}+m-2)\\
	&={2\tilde{m}_{i,l}}+{2\tilde{m}_{i,l+1}}-m+2, {e_{l,j}}-t_{l-1,j}\\
	&={2\tilde{m}_{l,j}}-(-2\tilde{m}_{l-1,j}+m-2)\\
	&={2\tilde{m}_{l,j}}+2\tilde{m}_{l-1,j}-m+2.
\end{align*} 
where the subscripts are taken modulo $m$. Therefore, 
\begin{align*}
\rho^{12}_{i}
	&=r^{o_{i,i}}r^{e_{i,i+1}}sr^{t_{i,i+2}}sr^{e_{i,i+3}}sr^{t_{i,i+4}}s\ldots r^{e_{i,i-2}}sr^{t_{i,i-1}}s\\
	&=r^{2\tilde{m}_{i,i}+1}r^{e_{i,i+1}}r^{-t_{i,i+2}}r^{e_{i,i+3}}r^{-t_{i,i+4}}\ldots r^{e_{i,i-2}}r^{-t_{i,i-1}}\\
	&=r^{2\tilde{\mu}-(m-2)(m-1)/2+1},\\
\sigma^{12}_{j}
	&=r^{o_{j,j}}r^{e_{j-1,j}}sr^{t_{j-2,j}}sr^{e_{j-3,j}}sr^{t_{j-4,j}}s\ldots r^{e_{j+2,j}}sr^{t_{j+1,j}}s\\
	&=r^{2\tilde{m}_{j,j}+1}r^{e_{j-1,j}}r^{-t_{j-2,j}}r^{e_{j-3,j}}r^{-t_{j-4,j}}\ldots r^{e_{j+2,j}}r^{-t_{j+1,j}}\\
	&=r^{2\tilde{\mu}-(m-2)(m-1)/2+1},\\
\rho^{21}_{i}
	&=r^{o_{i,i+1}}r^{e_{i,i+2}}sr^{t_{i,i+3}}sr^{e_{i,i+4}}sr^{t_{i,i+5}}s\ldots r^{e_{i,i-2}}sr^{t_{i,i-1}}s\\
	&=r^{2\tilde{m}_{i,i+1}+1}r^{e_{i,i+2}}r^{-t_{i,i+3}}r^{e_{i,i+4}}r^{-t_{i,i+5}}\ldots r^{e_{i,i-2}}r^{-t_{i,i-1}}=r^{2\tilde{\mu}-(m-2)(m-1)/2+1},\\
\sigma^{21}_{j}
	&=r^{o_{j,j+1}}r^{e_{j-1,j+1}}sr^{t_{j-2,j+1}}sr^{e_{j-3,j+1}}sr^{t_{j-4,j+1}}s\ldots r^{e_{j-2,j+1}}sr^{t_{j-1,j+1}}s\\
	&r^{2\tilde{m}_{j,j+1}+1}r^{e_{j-1,j+1}}r^{-t_{j-2,j+1}}r^{e_{j-3,j+1}}r^{-t_{j-4,j+1}}\ldots r^{e_{j-2,j+1}}r^{-t_{j-1,j+1}}\\
	&=r^{2\tilde{\mu}-(m-2)(m-1)/2+1}.
\end{align*}

An example is shown in Figure~\ref{fig:Q^21(5)}.   
   
\begin{figure}[H]
\begin{center}
%\begin{subfigure}{1\linewidth}
$$
\begin{array}{|c|c|c|c|c|}
\hline
r^{e_{1,2}}s & r^{t_{1,3}}s  & r^{e_{1,4}}s	& r^{t_{1,5}}s	&r^{o_{1,1}}\\ \hline
r^{o_{2,2}} & r^{e_{2,3}}s  & r^{t_{2,4}}s	& r^{e_{2,5}}s	&r^{t_{2,1}}s \\ \hline
r^{t_{3,2}}s & r^{o_{3,3}}  & r^{e_{3,4}}s	& r^{t_{3,5}}s	&r^{e_{3,1}}s \\ \hline
r^{e_{4,2}}s & r^{t_{4,3}}s  & r^{o_{4,4}}	& r^{e_{4,5}}s	&r^{t_{4,1}}s \\ \hline
r^{t_{5,2}}s & r^{e_{5,3}}s  & r^{t_{5,4}}s	& r^{o_{5,5}}	&r^{e_{5,1}}s  \\ \hline
\end{array}=
\begin{array}{|c|c|c|c|c|}
\hline
r^{2}s&r^{37}s&r^{30}s&r^{19}s&r^{49}	\\\hline
r^{15}&r^{28}s&r^{21}s&r^{46}s&r^{43}s	\\\hline
r^{27}s &r^{41}&r^{44}s&r^{45}s &r^{12}s\\\hline
r^{38}s&r^{11}s&r^{7}&r^{20}s &r^{29}s	\\\hline
r^{3}s&r^{4}s&r^{35}s&r^{13}	&r^{36}s\\\hline
\end{array}
$$
%\caption{$Q^{21}(5)$}
%\end{subfigure}
\end{center}
\caption{Square $Q^{21}(5)$}\label{fig:Q^21(5)}
\end{figure}

\end{const}
We are now ready to state our main result of this section.

\begin{theorem}\label{thm:odd}
	There exists a semi-circularly $D_{2m^2}$-semi-magic square $Q(2m)$ for every odd $m$, $m>1$. 
\end{theorem}

\begin{proof}
Let $Q^{u,v}$ be the four squares obtained by Constructions~\ref{const:Q11}, \ref{const:Q22}  and \ref{const:Q12Q21}
We will glue them to obtain a square $Q$ as in Figure~\ref{fig:Q}.
	
Each row product is now performed as
\begin{align*}	
\rho_i	
	&=	(q^{u1}_{i,i}\ q^{u1}_{i,i+1}\dots  q^{u1}_{i,m-1} q^{u1}_{i,0} q^{u1}_{i,1}\dots q^{u1}_{i,i-1})\ 
 					(q^{u2}_{i,i}\ q^{u2}_{i,i+1}\dots q^{u2}_{i,m-1}  q^{u2}_{i,0} q^{u2}_{i,1}\dots  q^{u2}_{i,i-1})\\
 	&= \rho^{u1}_i \rho^{u2}_i\\ &=r^{2\tilde{\mu}}r^{2\tilde{\mu}-(m-2)(m-1)/2+1}\\ &=r^{4\tilde{\mu}-(m-2)(m-1)/2+1}
\end{align*}
	and the column products as
\begin{align*}
\sigma_j	
	&=(q^{1v}_{j,j}\ q^{1v}_{j-1,j} \dots q^{1v}_{0,j}q^{1v}_{m-1,j}q^{1v}_{m-2,j} \dots q^{1v}_{j+1,j})\ 
 					(q^{2v}_{j,j}\ q^{2v}_{j-1,j} \dots q^{2v}_{0,j}q^{2v}_{m-1,j}q^{2v}_{m-2,j} \dots q^{2v}_{j+1,j})
	\\
	&= \sigma^{1v}_i \sigma^{2v}_i\\ &=r^{2\tilde{\mu}+m/2}r^{2\tilde{\mu}-(m-2)(m-1)/2+1}\\ &=r^{4\tilde{\mu}-(m-2)(m-1)/2+1}, 	  
\end{align*}
where the subscripts are taken modulo $m$.
	All row and column products in the square $Q(2m)$ are semi-circular and equal to the magic constant $\mu=r^{4\tilde{\mu}-(m-2)(m-1)/2+1}$, which completes the proof.
\end{proof}
\begin{exm}\label{exm:D(10)} In Figure~\ref{fig:D(10)} we show the construction of magic square $M_{D_{50}}(10)$ using power squares from Figure~\ref{fig:E(5)}.
\end{exm}

\begin{figure}

$$
\begin{array}{||c|c|c|c|c||c|c|c|c|c||}
\hline\hline
r^{34} & r^{48} & r^{2}  & r^{16} & r^{30} & r^{35}&r^{48}s&r^{1}s&r^{16}s&r^{23}s\\\hline
r^{46} & r^{10} & r^{14} & r^{28} & r^{32} &r^{7}s&r^{11}&r^{14}s&r^{25}s&r^{32}s\\\hline
r^{8}  & r^{12} & r^{26} & r^{40} & r^{44}& r^{8}s&r^{41}s &r^{27}&r^{40}s&r^{9}s \\\hline
r^{20} & r^{24} & r^{38} & r^{42} & r^{6} &r^{33}s&r^{24}s&r^{15}s&r^{43}&r^{6}s \\\hline
r^{22} & r^{36} & r^{0}  & r^{4}  & r^{18}&r^{22}s&r^{17}s&r^{0}s&r^{49}s&r^{19}\\ \hline\hline
r^{49}&r^{2}s&r^{37}s&r^{30}s&r^{19}s&r^{34}s &  r^{5}s& r^3  & r^{17} & r^{31} \\\hline
r^{43}s&r^{15}&r^{28}s &r^{21}s&r^{46}s&r^{47} &r^{10}s &r^{39}s  & r^{29} & r^{33} \\\hline
r^{12}s&r^{27}s&r^{41}&r^{44}s&r^{45}s&r^{9}  & r^{13} & r^{26}s &r^{13}s  & r^{45} \\\hline
r^{29}s&r^{38}s&r^{11}s&r^{7}&r^{20}s&r^{21} & r^{25} & r^{39} & r^{42}s & r^{47}s \\\hline
r^{36}s&r^{3}s&r^{4}s&r^{35}s&r^{23}& r^{31}s & r^{37} & r^{1}  & r^{5}  & r^{18}s\\ \hline\hline
 
\end{array}
$$
\caption{$SMS_{D_{50}}(10)$ with the magic constant $\mu=r^{5}$}\label{fig:D(10)}

\end{figure}
As in Section~\ref{sec:even}, we now analyze the case where $m$ is odd and show that the same situation occurs. Let the row products $\rho^{uv}_{i,c}$ and column $\sigma^{uv}_{i,c}$ products be defined as in Section~\ref{sec:even}.

Observe that in Construction~\ref{const:Q12Q21} (that is, for  $u+v$ odd) we have
\[
\rho^{uv}_{i,c}
=
\begin{cases}
r^{2\tilde{\mu}-(m-2)(m-1)/2+1} & \text{if } c \text{ is odd or}\;\; c=0,\\[4pt]
r^{-2\tilde{\mu}+(m-2)(m-1)/2-1} & \text{if } c \text{ is even and}\;\;c\neq 0,
\end{cases}
\]
\[
\sigma^{uv}_{j,c}
=
\begin{cases}
r^{2\tilde{\mu}-(m-2)(m-1)/2+1} & \text{if } c \text{ is odd or}\;\; c=0,\\[4pt]
r^{-2\tilde{\mu}+(m-2)(m-1)/2-1} & \text{if } c \text{ is even and}\;\;c\neq 0,
\end{cases}
\]

Consequently, in the proof of Theorem~\ref{thm:odd}, the resulting magic constant again depends on the chosen ordering. Indeed,
\[
\rho_i
=
\rho^{u1}_{i,0}\rho^{u2}_{i,c}
=
\begin{cases}
r^{4\tilde{\mu}-(m-2)(m-1)/2+1}, & \text{if } c \text{ is odd or}\;\; c=0,\\[4pt]
r^{(m-2)(m-1)/2-1}, & \text{if } c \text{ is even and}\;\;c\neq 0,
\end{cases}
\]
and similarly for the column products,
\[
\sigma_j
=
\sigma^{1v}_{j,0}\sigma^{{2}v}_{j,c}
=
\begin{cases}
r^{4\tilde{\mu}-(m-2)(m-1)/2+1}, & \text{if } c \text{ is odd or}\;\; c=0,\\[4pt]
r^{(m-2)(m-1)/2-1}, & \text{if } c \text{ is even and}\;\;c\neq 0.
\end{cases}
\]
Thus, also in the case where $m$ is odd, two distinct magic constants may arise from the same semi-magic square, depending on the order in which the products are taken.

%\input{70__02a}
%\input{80__02a}

%\newpage
\section{Conclusion}\label{sec:conclusion}

%We have shown that a magic square over a dihedral group must be of even side, and proved the existence of such semi-magic rectangles for all even $n\geq4$. We summarize our two previous results as follows.
{We summarize our two previous results below by stating the necessary and sufficient conditions for the existence of
$SMS_{D_{2k}}(n)$.}

\begin{theorem}\label{thm:main}
	There exists a $\Gamma$-semi-magic square $SMS_{\Gamma}(n)$, where $\Gamma$ is a dihedral group, {if and only if $n$ is even and} $n\geq4$.
\end{theorem}

\begin{proof} It can be verified easily that such $SMS_{\Gamma}(2)$ does not exist. 	The necessity follows from Theorem~\ref{thm:necessary}. The existence follows from Theorem~\ref{thm:meven} for $n\equiv0\pmod4$ and from Theorem~\ref{thm:odd} for $n\equiv2\pmod4$.
\end{proof}

Because we were unable to find  linearly $\Gamma$-semi-magic squares,  we pose an open problem here.

\begin{prb}\label{prb:linear}
	Construct linearly $\Gamma$-semi-magic squares $SMS_{\Gamma}(n)$, where $\Gamma$ is a dihedral group, for every  $n\geq4$.
\end{prb}

Finally, because all our squares are $\Gamma$-semi-magic but not $\Gamma$-magic, we conclude with the following.

\begin{prb}\label{prb:all-gamma-semimagic}
	Construct $\Gamma$-magic squares $MS_{\Gamma}(n)$, where $\Gamma$ is a dihedral group, for every even $n$, $n\geq4$.
	%$n\equiv0\pmod2$, $n\geq6$.
\end{prb}
\begin{exm}\label{exm:D(8)2} In Figure~\ref{fig:MRD2} we show a magic rectangle $M_{D_{4}}(2,4)=(a_{i,j})_{2\times 4}$. Namely, observe that
$$a_{1,1} a_{1,2} a_{1,3}a_{1,4}=a_{2,1} a_{2,2} a_{2,4}a_{2,3}=r^2$$ and
$$a_{1,1}a_{2,1}=a_{1,2}a_{2,2}=a_{1,3}a_{2,3}=a_{2,4}a_{1,4}=r^1.$$

\end{exm}
\begin{figure}[h!]
\begin{center}
\setlength\extrarowheight{2pt}
$
\begin{array}{|c|c|c|c|}
\hline
r^0& r^3& r^0s &r^1s\\\hline
r^1	& r^2&r^3s&r^2s\\\hline
\end{array}
$
\end{center}
\caption{A $D_4$-semi-magic rectangle $MR_{D_4}(2,4)$}
\label{fig:MRD2}
\end{figure}

Motivated by Example~\ref{exm:D(8)2}, we also state the following open problem.
\begin{prb}\label{prb:all-gamma-magic}
	Characterize  $\Gamma$-magic rectangles $MS_{\Gamma}(m,n)$, where $\Gamma$ is a dihedral group of order $mn/2$.
	%$n\equiv0\pmod2$, $n\geq6$.
\end{prb}

In our constructions, we showed that it is possible to obtain different magic constants for a given magic rectangle $SMS_{D_{2m^2}}(2m)$. Some of these constants depend on $\tilde{\mu}$, which is the magic constant of a numerical magic square $\tilde{M}(m)$. Observe that the square $\tilde{M}(m)$ with the magic constant
$\tilde{\mu}$ can be transformed into a magic square $\widehat{M}(m)$ with a different magic constant $\widehat{\mu}$ via a uniform translation. Specifically, adding $x$ to each entry of $\tilde{M}(m)$ produces $\widehat{M}(m)$ with
\[
\widehat{\mu} = \tilde{\mu} + m x.
\]

\begin{exm}\label{exm:D(10)2}
In Figure~\ref{fig:D(10)2}, the magic square $\widehat{M}(5)$ has a magic constant
$\widehat{\mu}=70$. Therefore, we obtain an $SMS_{D_{50}}(10)$ with a magic constant $\mu=r^{25}$.
\end{exm}

\begin{figure}[H]
\begin{subfigure}{1\linewidth}
$$
\begin{array}{|c|c|c|c|c|}
\hline
18 & 25 & 2  & 9  & 16 \\\hline
24 & 6  & 8  & 15 & 17 \\\hline
5  & 7  & 14 & 21 & 23 \\\hline
11 & 13 & 20 & 22 & 4  \\\hline
12 & 19 & 26 & 3  & 10 \\\hline
\end{array}
$$
\caption{$\widehat{M}(5)$, $\widehat{\mu}=70$}
\end{subfigure}\\
\hfill
\begin{subfigure}{0.33\linewidth}
$$
\begin{array}{|c|c|c|c|c|}
\hline
36 & 0  & 4  & 18 & 32 \\\hline
48 & 12 & 16 & 30 & 34 \\\hline
10 & 14 & 28 & 42 & 46 \\\hline
22 & 26 & 40 & 44 & 8  \\\hline
24 & 38 & 2  & 6  & 20 \\\hline
\end{array}
$$
\caption{$E(5)$}
\end{subfigure}
\begin{subfigure}{0.33\linewidth}
$$
\begin{array}{|c|c|c|c|c|}
\hline
37 & 1  & 5  & 19 & 33 \\\hline
49 & 13 & 17 & 31 & 35 \\\hline
11 & 15 & 29 & 43 & 47 \\\hline
23 & 27 & 41 & 45 & 9  \\\hline
25 & 39 & 3  & 7  & 21 \\\hline
\end{array}
$$
\caption{$O(5)$}
\end{subfigure}
\begin{subfigure}{0.32\linewidth}
$$
\begin{array}{|c|c|c|c|c|}
\hline
17 & 3  & 49 & 35 & 21 \\\hline
5  & 41 & 37 & 23 & 19 \\\hline
43 & 39 & 25 & 11 & 7  \\\hline
31 & 27 & 13 & 9  & 45 \\\hline
29 & 15 & 1  & 47 & 33 \\\hline
\end{array}
$$
\caption{$T(5)$}
\end{subfigure}
\begin{subfigure}{0.9\linewidth}
$$
\begin{array}{||c|c|c|c|c||c|c|c|c|c||}
\hline\hline
r^{36} & r^{0} & r^{4}  & r^{18} & r^{32} & r^{37}&r^{0}s&r^{49}s&r^{18}s&r^{21}s\\\hline
r^{48} & r^{12} & r^{16} & r^{30} & r^{34} & r^{5}s&r^{13}&r^{16}s&r^{23}s&r^{34}s\\\hline
r^{10}  & r^{14} & r^{28} & r^{42} & r^{46} & r^{10}s&r^{39}s &r^{29}&r^{42}s&r^{7}s \\\hline
r^{22} & r^{26} & r^{40} & r^{44} & r^{8} & r^{31}s&r^{26}s&r^{13}s&r^{45}&r^{8}s \\\hline
r^{24} & r^{38} & r^{2}  & r^{6}  & r^{20} & r^{24}s&r^{15}s&r^{2}s&r^{47}s&r^{21}\\ \hline\hline
r^{1} & r^{4}s & r^{35}s & r^{32}s & r^{17}s & r^{36}s & r^{3}s & r^5  & r^{19} & r^{33} \\\hline
r^{41}s & r^{17} & r^{30}s & r^{19}s & r^{48}s & r^{49} & r^{12}s & r^{37}s & r^{31} & r^{35} \\\hline
r^{14}s & r^{25}s & r^{43} & r^{46}s & r^{43}s & r^{11}  & r^{15} & r^{28}s & r^{11}s  & r^{47} \\\hline
r^{27}s & r^{40}s & r^{9}s & r^{9} & r^{22}s & r^{23} & r^{27} & r^{39} & r^{44}s & r^{45}s \\\hline
r^{38}s & r^{1}s & r^{6}s & r^{33}s & r^{25} & r^{29}s & r^{39} & r^{3} & r^{7} & r^{20}s\\ \hline\hline
\end{array}
$$
\caption{$SMS_{D_{50}}(10)$ with magic constant $\mu=r^{25}$}
\end{subfigure}

\caption{Example of $E(5)$, $O(5)$, $T(5)$, and $SMS_{D_{50}}(10)$}\label{fig:D(10)2}
\end{figure}

Thus, for a dihedral group $\Gamma$ of even order $n^2>4$, we can define the
\emph{spectrum of magic constants} by
\[
\Spec_{\Gamma}(n)
=
\left\{
\mu \in \Gamma \;:\;
\text{there exists an $SMS_{\Gamma}(n)$ with magic constant $\mu$}
\right\}.
\]

Consequently, we pose the following open problem: 

\begin{prb}\label{prb:spectrum}
	Characterize $\Spec_{\Gamma}(n)$, where $\Gamma$ is a dihedral group of order $n^2>4$.
\end{prb}

\section{Statements and declarations}

The work of the first author was    supported by program AGH University of Krakow under grant no. 16.16.420.054, funded by the Polish Ministry of Science and Higher Education. The work of the second author was partially supported by program ''Excellence initiative – research university'' for the AGH University.

\vskip1cm

\noindent

%\newpage

%\newpage

\begin{comment}

\bibitem{Cichacz-Hincz-1}
S. Cichacz, T. Hinc,
A note on magic rectangle set $MRS_{\Gamma}(2k+1,4;4l+2)$,
\emph{J. Combin. Des.} \textbf{29(7)} (2021), 502--507.

\end{comment}



% \end{thebibliography}

\end{document}